\documentclass[12pt,twoside]{amsart}
\usepackage{latexsym,amsmath,amsopn,amssymb,amsthm,amsfonts}
\usepackage[T2A]{fontenc}
\usepackage[cp1251]{inputenc}
\usepackage[english]{babel}
\usepackage{graphicx}

\newtheorem{theorem}{Theorem}[section]

\newtheorem{statement}{Statement}[section]

\begin{document}
	
\vspace{0.5cm}
\title[Hereditary classes of structures: study of their universal theories]{Axiomatizability of hereditary classes of structures of finite and infinite languages and decidability of their universal theories}
\author{A.~V.~Ilev}
\thanks{The work was carried out within the framework of the State Contract to the IM SB RAS, project FWNF--2022--0003.}
\address{Artem Victorovich Ilev,
\newline\hphantom{II} Sobolev Institute of Mathematics of the SB RAS,
\newline\hphantom{II} 13 Pevtsova st., Omsk, 644099, Russia}
\email{artyom\_iljev@mail.ru}

\begin{abstract}
In the paper hereditary classes of ${\rm L}$\nobreakdash-structures are studied with 
\linebreak
language of the form ${{\rm L} = {\rm L_{fin}} \cup {\rm L_\infty}}$, where ${{\rm L_{fin}} = \langle R_1,R_2,\ldots, R_m, = \rangle}$ and 
\linebreak
${{\rm L_\infty} = \langle R_{m+1}, R_{m+2}, \ldots \rangle}$, and also in~${\rm L_\infty}$ the number of predicates of each arity is finite, all predicates are ordered in ascending of their arities and satisfying the property of non\nobreakdash-repetition of elements.
A~class of ${\rm L}$\nobreakdash-structures is called {\it hereditary} if it is closed under substructures.
It is proved that class of ${\rm L}$\nobreakdash-structures is hereditary if and only if it can be defined in terms of forbidden substructures.
A~class of ${\rm L}$\nobreakdash-structures is called {\it universal axiomatizable} if there is a set~$Z$ of universal ${\rm L}$\nobreakdash-sentences such that the class consists of all structures that satisfy~$Z$.
The problems of universal axiomatizability of hereditary classes of ${\rm L}$\nobreakdash-structures are considered in the paper.
It is shown that hereditary class of ${\rm L}$\nobreakdash-structures is universal axiomatizable if and only if it can be defined in terms of finite forbidden substructures.
It is proved that the universal theory of any axiomatizable hereditary class of ${\rm L}$\nobreakdash-structures with recursive set of minimal forbidden substructures is decidable.
\vspace{3mm}

\end{abstract}
\maketitle

\section{Preliminary information}

Let's remember basic definitions of model theory.
All necessary notions and results can be also found in~\cite{ep79, m02}.

\medskip

{\it Language} (or {\it signature}) ${{\rm L}={\rm R} \cup {\rm F} \cup {\rm C}}$ is a~sequence of the following sets:

(1) set of {\it predicate symbols}~${\rm R}$;

(2) set of {\it functional symbols}~${\rm F}$;

(3) set of {\it constant symbols}~${\rm C}$;
\\
where each predicate symbol ${R \in {\rm R}}$ and each function symbol ${F \in {\rm F}}$ are associated with a natural numbers~$n_R$ and~$n_F$ correspondently, which are called {\it arities}.

An {\it ${\rm L}$\nobreakdash-structure} (or {\it algebraic system of language~${\rm L}$}) is a~sequence of the form
$${\mathcal A} = \langle A;\ R^{\mathcal A}, F^{\mathcal A}, c^{\mathcal A}\rangle,$$
where $A$~--- is a~non-empty set called {\it the basic set} of the structure~$\mathcal A$;
each predicate symbol ${R \in {\rm R}}$ corresponds to a $n_R$\nobreakdash-ary relation ${R^{\mathcal A} \subseteq A^{n_R}}$;
each function symbol ${F \in {\rm F}}$ corresponds to a $n_F$\nobreakdash-ary function ${F^{\mathcal A} : A^{n_F} \rightarrow A}$;
each constant symbol ${c \in {\rm C}}$ corresponds to some element ${c^{\mathcal A} \in A}$.
Further we will use the short notation ${{\mathcal A} = \langle A, {\rm L} \rangle}$ for describing ${\rm L}$\nobreakdash-structures.
An ${\rm L}$\nobreakdash-structure~$\mathcal A$ is called a~{\it model} if it doesn't contain functions.

Two ${\rm L}$\nobreakdash-structures ${{\mathcal A} = \langle A, {\rm L} \rangle}$ and ${{\mathcal B} = \langle B, {\rm L} \rangle}$ are called {\it isomorphic} if there is an isomorphism ${f: A \rightarrow B}$ preserving their predicates and functions.

An ${\rm L}$\nobreakdash-structure ${{\mathcal A}=\langle A, {\rm L} \rangle}$ is called a {\it substructure} of ${\rm L}$\nobreakdash-structure ${{\mathcal B}=\langle B, {\rm L} \rangle}$ if

(1) ${A \subseteq B}$;

(2) functions and predicates in~${\mathcal A}$ are the restrictions on~$A$ of corresponding functions and predicates in~${\mathcal B}$;

(3) the set~$A$ is closed under functions.

A~{\it formula of the language~${\rm L}$} is a~formula for the first-order predicate calculus with equality, the non-logical constants of which are contained in~${\rm L}$.
A~formula without free variables is called a {\it sentence}.
The truth of a sentence $\varphi$ in the ${\rm L}$\nobreakdash-structure~${\mathcal A}$ is denoted by ${{\mathcal A} \models \varphi}$.
A~sentence~$\varphi$ is called a~{\it universal sentence} or a~{\it $\forall$\nobreakdash-sentence} if ${\varphi = \forall x_1\ldots\forall x_n\ \psi}$, where $\psi$~--- quantifier-free formula containing no variables except ${x_1,\ldots,x_n}$.
A~sentence~$\varphi$ is called an~{\it existential sentence} or an~{\it $\exists$\nobreakdash-sentence} if ${\varphi = \exists x_1\ldots\exists x_n\ \psi}$, where $\psi$~--- quantifier-free formula containing no variables except ${x_1,\ldots,x_n}$.

\medskip

An~{\it abstract class of ${\rm L}$\nobreakdash-structures} is a~family of ${\rm L}$\nobreakdash-structures which is closed under isomorphism.
We will consider only abstract classes.
A~class of ${\rm L}$\nobreakdash-structures is called {\it hereditary} if it is closed under substructures.

A~class~$\textbf{\em K}$ of ${\rm L}$\nobreakdash-structures is called {\it axiomatizable} if there is a~set~$Z$ of ${\rm L}$\nobreakdash-sentences such that~$\textbf{\em K}$ consists of all ${\rm L}$\nobreakdash-structures that satisfy~$Z$.
The set~$Z$ is said to be a~{\it set of axioms} for the class~$\textbf{\em K}$.
If the set~$Z$ is finite, then~$\textbf{\em K}$ is called {\it finitely axiomatizable}.
If the set~$Z$ consists only of universal sentences, then~$\textbf{\em K}$ is called {\it universal axiomatizable}.
A~class~$\textbf{\em K}$ is called {\it recursive axiomatizable} if it can be defined by a~{\it recursive set of axioms}, i.\,e., there is an algorithm which correctly decides whether any ${\rm L}$\nobreakdash-sentence belongs to the set~$Z$.

The ${\rm L}$\nobreakdash-sentences $\varphi_1$ and $\varphi_2$ are called {\it equivalent} in the class~$\textbf{\em K}$ of ${\rm L}$\nobreakdash-strucrures, if for any structure~${\mathcal A}$ of the class~$\textbf{\em K}$
$${\mathcal A} \models \varphi_1 \Leftrightarrow {\mathcal A} \models \varphi_2.$$

Let~$\textbf{\em K}$ be a~class of ${\rm L}$\nobreakdash-structures.
The {\it (elementary) theory of a~class~$\textbf{K}$} is the set $Th(\textbf{\em K})$ of all ${\rm L}$\nobreakdash-sentences, which are true in all structures from~$\textbf{\em K}$.
A~theory is called {\it decidable} if there is an algorithm that for any ${\rm L}$\nobreakdash-sentence determines whether or not this sentence belongs to the theory.
A~set of all $\forall$\nobreakdash-sentences of $Th(\textbf{\em K})$ is called the {\it universal theory} or {\it $\forall$\nobreakdash-theory of the class~$\textbf{K}$}.
A~set of all $\exists$\nobreakdash-sentences of $Th(\textbf{\em K})$ is called the {\it existential theory} or {\it $\exists$\nobreakdash-theory of the class~$\textbf{K}$}.

Let $\textbf{\em H}$ be a~set of ${\rm L}$\nobreakdash-structures.
Then the class $Forb(\textbf{\em H})$, which consists of all ${\rm L}$\nobreakdash-structures that don't contain substrucrures from~$\textbf{\em H}$, is an abstract class, i.\,e. it is closed under isomorphism.
This class can be defined by using ${\rm L}$\nobreakdash-structures ${{\mathcal A} \in \textbf{\em H}}$ as {\it forbidden substructures}.
We will say that the class~$\textbf{\em K}$ of ${\rm L}$\nobreakdash-structures {\it can be defined in terms of forbidden substructures} if ${\textbf{\em K} = Forb(\textbf {\em H})}$ for some set~$\textbf{\em H}$.

A~set of ${\rm L}$\nobreakdash-structures~$\textbf{\em H}$ is called a~{\it set of minimal forbidden substructures} for the class~$\textbf{\em K}$, if ${\textbf{\em K} = Forb(\textbf{\em H})}$ and also for any ${\rm L}$\nobreakdash-structure ${{\mathcal A} \in \textbf{\em H}}$ every its substructure ${{\mathcal A}_1 \not \in \textbf{\em H}}$.

\begin{statement}
\label{th10}
Let ${\textbf{K} = Forb(\textbf{H})}$.
A~set~$\textbf{H}$ is a~set of minimal forbidden substructures for the class~$\textbf{K}$ if and only if~$\textbf{H}$ is an inclusion-minimal set of forbidden substructures for the class~$\textbf{K}$, i.\,e. ${\textbf{K} \not = Forb(\textbf{H}_1)}$ for all ${\textbf{H}_1 \subset \textbf{H}}$.
\end{statement}

\begin{proof}{\it Necessity.}
Let's assume the opposite, that there is a~set ${\textbf{\em H}_1 \subset \textbf{\em H}}$ such that ${\textbf{\em K} = Forb(\textbf{\em H}_1)}$.
Then there is ${\rm L}$\nobreakdash-structure ${{\mathcal A} \in \textbf{\em H} \setminus \textbf{\em H}_1}$ and also ${{\mathcal A} \not \in Forb(\textbf{\em H}_1)}$, i.\,e. there is ${\rm L}$\nobreakdash-structure ${{\mathcal A}_1 \in \textbf{\em H}_1}$ such that ${\mathcal A}_1$~is a~substructure of~${\mathcal A}$.
But since ${{\mathcal A}_1 \in \textbf{\em H}}$, then there is a~contradiction with the definition of the set of minimal forbidden substructures for the class~$\textbf{\em K}$.

\medskip

{\it Sufficiency.}
Let's assume the opposite, that there is a~${\rm L}$\nobreakdash-structure ${{\mathcal A} \in \textbf{\em H}}$ and 
\linebreak
its substructure ${{\mathcal A}_1 \in \textbf{\em H}}$.
Consider the set ${\textbf{\em H}_1 = \textbf{\em H} \setminus \{{\mathcal A}\}}$.
It's obvious that 
\linebreak
${\textbf{\em K} = Forb(\textbf{\em H}_1)}$, because all ${\rm L}$\nobreakdash-structures which don't contain~${\mathcal A}$ as a~substructure mustn't contain ${{\mathcal A}_1 \in \textbf{\em H}_1}$ as a~substructure.
I.\,e. there is a~contradiction with the condition of the statement.
\end{proof}

A~set of forbidden substructures of a~language~${\rm L}$ is called {\it recursive} if there exists a~Godel numbering~$g$ of these substructures such that the set of their numbers is recursive, i.\,e. there is an algorithm that for any natural number will output <<YES>> if it is belongs to the set of these numbers and will output <<NO>> otherwise.

\begin{statement}[Criterion for universal axiomatizability]{\rm \cite{ep79}}
\label{th11} 
Let~$\textbf{K}$ be an axioma\-tizable class of ${\rm L}$\nobreakdash-structures.
A~class~$\textbf{K}$ is universal axiomatizable if and only if it is closed under substructures.
\end{statement}

In this paper we consider only two types of structures.
Firstly, models ${{\mathcal A} = \langle A, {\rm L_{fin}} \rangle}$ of finite languages with equality, in which ${{\rm L_{fin}} = \langle R_1,R_2,\ldots, R_m, = \rangle}$.
Secondly, models ${{\mathcal A} = \langle A, {\rm L} \rangle}$ of infinite languages with equality,  in which ${{\rm L} = {\rm L_{fin}} \cup {\rm L_\infty}}$, where ${{\rm L_\infty} = \langle R_{m+1}, R_{m+2}, \ldots \rangle}$, and in~${\rm L_\infty}$ the number of predicates of each arity is finite, all predicates are ordered in ascending of their arities and satisfying the property of non\nobreakdash-repetition of elements, i.\,e. for all ${R_k \in {\rm L_\infty}}$:

\medskip

$\bullet$ $\forall x_1 \ldots \forall x_l\ [R_k(x_1,\ldots,x_l) \rightarrow \bigwedge\limits_{i \not = j} (x_i \not = x_j)]$.
\\
We denote by~$n_k$ the arity of the corresponding predicate~$R_k$.

\medskip

Since ${{\rm L_{fin}} \subset {\rm L}}$ for all ${m \in {\mathbb N}}$, then the finite case will not be mentioned specially in statements and their proofs, which are true for any considered language~${\rm L}$.

\section{Axiomatizable hereditary classes} 

We consider some theorems, which is necessary to indicate the connection between hereditary classes of ${\rm L}$-structures and their forbidden substructures.

\begin{theorem}
\label{th21}
An abstract class~$\textbf{K}$ of ${\rm L}$\nobreakdash-structures is hereditary if and only if it can be defined in terms of forbidden substructures.
\end{theorem}

\begin{proof}{\it Necessity.}
Let $\textbf{\em K}$~be a~hereditary class of ${\rm L}$\nobreakdash-structures, i.\,e. for any ${\rm L}$\nobreakdash-struc\-tures~${\mathcal A}_1$ and~${\mathcal A}_2$ if ${{\mathcal A}_1 \in \textbf{\em K}}$ and ${\mathcal A}_2$ is an arbitrary substructure of ${\mathcal A}_1$, then ${{\mathcal A}_2 \in \textbf{\em K}}$.
We consider the class $\textbf{\em H}$~--- the addition to the class~$\textbf{\em K}$ in the class of all ${\rm L}$\nobreakdash-structures.
Since ${{\mathcal A}_2 \not \in \textbf{\em H}}$, then ${{\mathcal A}_1 \in Forb(\textbf{\em H})}$ and therefore ${\textbf{\em K} \subseteq Forb(\textbf{\em H})}$.

Now we consider an arbitrary ${\rm L}$\nobreakdash-structure ${{\mathcal A}_3 \in Forb(\textbf{\em H})}$, i.\,e. every substructure of which, including itself~${\mathcal A}_3$, isn't contained in~$\textbf{\em H}$.
But then ${{\mathcal A}_3 \in \textbf{\em K}}$ and, therefore, ${Forb(\textbf{\em H}) \subseteq \textbf{\em K}}$.

Thus, ${\textbf{\em K} = Forb(\textbf{\em H})}$, i.\,e. the class~$\textbf{\em K}$ of ${\rm L}$\nobreakdash-structures can be defined in terms of forbidden substructures.

\medskip

{\it Sufficiency.}
Let ${\textbf{\em K} = Forb(\textbf{\em H})}$~be a~class of ${\rm L}$\nobreakdash-structures, which can be defined in terms of forbidden substructures.
Let's assume the opposite, that there is ${\rm L}$\nobreakdash-structure ${{\mathcal A}_1 \in \textbf{\em K}}$ and its substructure~${\mathcal A}_2$ such that ${{\mathcal A}_2 \not \in \textbf{\em K}}$.
Then ${\rm L}$\nobreakdash-structure~${\mathcal A}_2$ contains substructure~${\mathcal A}_3$ such that ${{\mathcal A}_3 \in \textbf{\em H}}$.
But since~${\mathcal A}_3$ is also a~substructure of~${\mathcal A}_1$, therefore, ${{\mathcal A}_1 \not \in Forb(\textbf{\em H})}$.
It is a~contradiction.

Thus, for any ${\rm L}$\nobreakdash-structure ${{\mathcal A}_1 \in \textbf{\em K}}$ every its substructure contains in~$\textbf{\em K}$.
Therefore, $\textbf{\em K}$~is hereditary class of ${\rm L}$\nobreakdash-structures.
\end{proof}

\begin{theorem}
\label{th22}
Let ${\textbf{K} = Forb(\textbf{H})}$, and all ${\rm L}$\nobreakdash-structures of the set~$\textbf{H}$ are finite; ${\mathcal A}$ is an infinite ${\rm L}$\nobreakdash-structure, in which every finite substructure belongs to the class~$\textbf{K}$.
Then~${\mathcal A}$ is also belongs to the class~$\textbf{K}$.
\end{theorem}

\begin{proof}
Let's assume the opposite, that ${{\mathcal A} \not \in \textbf{\em K}}$.
Then there is its substructure~${\mathcal B}$ such that ${{\mathcal B} \in \textbf{\em H}}$.
But due to the condition of the theorem ${\mathcal B}$~must be finite~--- it is a~contradiction with that all finite substructures of~${\mathcal A}$ belong to the class~$\textbf{\em K}$ and, therefore, not belong to the set~$\textbf{\em H}$.

Thus, ${{\mathcal A} \in \textbf{\em K}}$.
\end{proof}

\begin{theorem}
\label{th23}
A~hereditary class of ${\rm L}$\nobreakdash-structures is (universal) axiomatizable if and only if it can be defined in terms of finite forbidden substructures.
\end{theorem}

\begin{proof}{\it Necessity.}
By definition, the hereditary class of ${\rm L}$\nobreakdash-structures is closed under substructures, therefore, due to the criterion of universal axiomatizability~\ref{th11} \cite[p.~165, Theorem~5]{ep79} any axiomatizable hereditary class of ${\rm L}$\nobreakdash-structures is $\forall$\nobreakdash-axioma\-tizable, therefore, every its axiom can be considered as a~$\forall$\nobreakdash-sentence.

Then the set of forbidden substructures of the hereditary class~$\textbf{\em K}$, which exists by Theorem~\ref{th21}, can be defined in the following way.

For every axiom~$\varphi$ we define a~finite set~$\textbf{\em H}_\varphi$ of forbidden substructures with number of elements from $1$~to~$p$, where~$p$ is a~number of variables in this axiom.
Its negation~$\neg \varphi$ is equivalent to the sentence ${\exists x_1\ldots\exists x_p\ \psi}$, where $\psi$ is a~quantifier-free formula in prenex disjunctive form (PDF), i.\,e. ${\psi = \psi_1 \vee \ldots \vee \psi_r}$, where ${\psi_1, \ldots, \psi_r}$ are conjuncts.
If any of that conjuncts doesn't contains the factors ${(x_i = x_j)}$ and ${(x_i \not = x_j)}$, then it is supplemented by the condition ${(x_i = x_j) \vee (x_i \not = x_j)}$.
Similarly for all ${R_k \in {\rm L_{fin}}}$ and finite number of predicates ${R_k \in {\rm L_\infty}}$, arity of which doesn't exeed~$p$, if any conjunct doesn't contains the factors $R_k(t_1,\ldots,t_l)$ and $\neg R_k(t_1,\ldots,t_l)$, then it is supplemented by the condition ${R_k(t_1,\ldots,t_l) \vee \neg R_k(t_1,\ldots,t_l)}$ for all 
\linebreak
${\{t_1,\ldots,t_l\}\subseteq \{x_1,\ldots,x_p\}}$, where ${l=n_k}$.
As a~result, we can go to a~sentence in PDF, that equivalent to $\neg \varphi$, each conjunct of which will either define a~substructure of the language~${\rm L}$ with a~fixed number of elements from $1$~to~$p $ and all possible fixed sets of elements that satisfy or not satisfy the predicates~$R_k$ of the language~${\rm L}$, or will contradict the restrictions imposed on the ${\rm L}$\nobreakdash-structures and must be excluded from consideration.

Then, we unify the sets~$\textbf{\em H}_\varphi$ for all axioms~$\{\varphi\}$ and obtain a~family~$\textbf{\em H}$ of finite forbidden substructures of the language~${ \rm L}$ for a~given hereditary class~$\textbf{\em K}$.

\medskip

{\it Sufficiency.}
Any finite ${\rm L}$\nobreakdash-structure can be associated with a~condition for the existence of a~substructure, which isomorphic to this structure.
It looks like 
\linebreak
${\varphi = \exists x_1 \ldots \exists x_p\ \psi}$, where $p$ is the number of elements of this ${\rm L}$\nobreakdash-structure, and $\psi$ is a~conjunct that contains conditions for the pairwise difference of all variables ${x_1,\ldots,x_p}$, while for all predicates ${R_k \in {\rm L_{fin}}}$ and a~finite number of predicates 
\linebreak
${R_k \in {\rm L_\infty}}$, the arity of which does not exceed~$p$, as well as all possible sets
\linebreak
${\{t_1,\ldots,t_l\}\subseteq \{x_1,\ldots,x_p\}}$, where ${l=n_k}$, the conjunct contains a~factor $R_k(t_1,\ldots,t_l)$ or $\neg R_k(t_1,\ldots,t_l)$.

We consider an arbitrary hereditary class ${\textbf{\em K} = Forb(\textbf{\em H})}$ of ${\rm L}$\nobreakdash-structures, where $\textbf{\em H}$ is a~set of finite forbidden substructures.
Then the axiomatics of the class~$\textbf{\em K}$ must consist of the set of axioms, each of which corresponds to one of the forbidden substructures from the set $\textbf{\em H}$, i.\,e. every axiom is the negations of the corresponding sentence ${\varphi = \exists x_1 \ldots \exists x_p\ \psi}$.

And due to the theorem~\ref{th22} such axioms will be sufficient to identify not only finite ${\rm L}$\nobreakdash-structures, which belong to the class~$\textbf{\em K}$, but also infinite ones.
Thus, any ${\rm L}$\nobreakdash-structure that satisfies the set of axioms~$\{\neg\varphi\}_{\textbf{\em H}}$ is contained in the hereditary class~$\textbf{\em K}$, and all of these axioms are $\forall$\nobreakdash-sentences, i.\,e. the class~$\textbf{\em K}$ is universally axiomatizable.
\end{proof}

\section{Decidability of universal theories of hereditary classes}

Establishing decidability of a~theory of any class~$\textbf{\em K}$ of ${\rm L}$\nobreakdash-structures makes it possible in principle to get an exhaustive list of properties of structures from this class.
Since decidable theories in complete form are quite rare, then obtaining a~proof of the decidability of universal theories and constructing a~corresponding algorithms is very actual problem.

\begin{theorem}
\label{th24}
The universal theory of any axiomatizable hereditary class of ${\rm L}$\nobreakdash-struc\-tures with recursive set of minimal forbidden substructures is decidable.
\end{theorem}

\begin{proof}
Due to the theorem~\ref{th23} we deal the case, where $Th_{\forall}(\textbf{\em K})$ is universal theory of arbitrary hereditary class~$\textbf{\em K}$ of ${\rm L}$\nobreakdash-structures, which defined in terms of finite forbidden substructures.
The following algorithm decides whether any universal ${\rm L}$\nobreakdash-sentence belongs to $Th_{\forall}(\textbf{\em K})$.

\bigskip

{\bf Algorithm for ${\rm L}$\nobreakdash-structures.}

\medskip

The input to the algorithm is an arbitrary universal sentence~$\varphi$.
Its negation~$\neg \varphi$ is transformed into a~sentence in a~prenex disjunctive form~(PDF) that is equivalent to~$\neg \varphi$ in the class of all ${\rm L}$\nobreakdash-structures.
The algorithm tries to construct an ${\rm L}$\nobreakdash-structure of the class~$\textbf{\em K}$ in which the sentence~$\neg \varphi$ will be true.
If it succeeds, then the sentence~$\varphi$ doesn't belong to the universal theory $Th_{\forall}(\textbf{\em K})$, and the algorithm outputs <<NO>>.
Otherwise, $\varphi$~belongs to this universal theory, and the algorithm outputs <<YES>>.

It is assumed that at each step the algorithm removes from the conjuncts of the current sentence all repeating factors when they occur.
This procedure has no effect on the equivalence of the sentense in ${\rm L}$\nobreakdash-structures and further we will not focus on it due to its naturalness.

\medskip

{\it Step~1.}
The algorithm formulates the sentence ${\neg \varphi = \exists x_1 \ldots \exists x_p\ \psi}$, where $\psi$~is a~quantifier free formula.
Then~$\neg \varphi$ is transformed to the equivalent sentence 
\linebreak
${\neg \varphi_1 = \exists x_1 \ldots \exists x_p\ \bigvee \limits_{r}\psi_r}$ in PDF, where each~$\psi_r$ is a~conjunct.

\medskip

{\it Step~2.}
The algorithm looks sequentially all conjuncts of the sentence~$\neg \varphi_1$.
If any conjunct~$\psi_r$ contains the variables $x_i$~and~$x_j$ but doesn't contain the factor ${(x_i = x_j)}$ or ${(x_i \not = x_j)}$, then the algorithm replaces the conjunct~$\psi_r$ by the disjunction ${[\psi_r \wedge (x_i=x_j)] \vee [\psi_r \wedge (x_i \not = x_j)]}$.
This procedure continues as long as possible.
Thus, we obtain an equivalent sentence~$\neg \varphi_2$, in each conjunct of which all its variables will be connected by equalities or inequalities.

\medskip

{\it Step~3.}
The algorithm looks sequentially all conjuncts of the sentence~$\neg \varphi_2$.
If any conjunct~$\psi_r$ contains the factor ${(x_i = x_j)}$, then the algorithm replaces the variable $x_j$~by~$x_i$ in the conjunct.
If any conjunct contains factors of the form ${(t = t)}$ where ${t \in \{x_1,\ldots,x_p\}}$, then these factors are removed as redundant.
If any conjunct contains factor of the form ${(t \not = t)}$, then such conjunct is removed from the sentence as identically false.
This procedure continues until all equalities are eliminated from all conjuncts.
Thus, we obtain an equivalent sentence~$\neg \varphi_3$ in PDF, in which each conjunct contains conditions for the pairwise difference of all its variables.

\medskip

{\it Step~4.}
The algorithm looks sequentially all conjuncts of the sentence~$\neg \varphi_3$.
For all predicates ${R_k \in {\rm L_{fin}}}$ if any conjunct~$\psi_r$ contains the variables ${t_1,\ldots,t_l}$, where ${t_i \in \{x_1,\ldots,x_p\}}$, but doesn't contain the factor $R_k(t_1,\ldots,t_l)$ or $\neg R_k(t_1,\ldots,t_l)$, 
\linebreak
then the algorithm replaces the conjunct~$\psi_r$ by the disjunction 
\linebreak
${[\psi_r \wedge R_k(t_1,\ldots,t_l)] \vee [\psi_r \wedge \neg R_k(t_1,\ldots,t_l)]}$.
This procedure continues as long as possible for all sets of repeatable variables ${\{t_1,\ldots,t_l\}\subseteq \{x_1,\ldots,x_p\}}$ of the conjunct, where ${l=n_k}$.
Moreover, if any conjunct contains the factors $R_k(t_1,\ldots,t_l)$ and $\neg R_k(t_1,\ldots,t_l)$ at the same time, then such conjunct is removed from the sentence as identically false.
Thus, we obtain an equivalent sentence~$\neg \varphi_4$ in PDF, in which each conjunct contains conditions for satisfaction or dissatisfaction of all sets of its variables to all predicates~${R_k \in {\rm L_{fin}}}$.

\medskip

{\it Step~5.}
The algorithm looks sequentially all conjuncts of the sentence~$\neg \varphi_4$.
For all predicates ${R_k \in {\rm L_\infty}}$, arity of which doesn't exeed~$p$, if any conjunct~$\psi_r$ contains the variables ${t_1,\ldots,t_l}$, where ${t_i \in \{x_1,\ldots,x_p\}}$, but doesn't contain the factor $R_k(t_1,\ldots,t_l)$ or $\neg R_k(t_1,\ldots,t_l)$, then the algorithm replaces the conjunct~$\psi_r$ by the disjunction ${[\psi_r \wedge R_k(t_1,\ldots,t_l)] \vee [\psi_r \wedge \neg R_k(t_1,\ldots,t_l)]}$.
This procedure continues as long as possible for all non-repeating sets of variables ${\{t_1,\ldots,t_l\}\subseteq \{x_1,\ldots,x_p\}}$ of the conjunct, where ${l=n_k}$.
Moreover, for all predicates ${R_k \in {\rm L_\infty}}$ if any conjunct contains the factor of the form $R_k(t_1,...,t,...,t,...,t_l)$ or the factors $R_k(t_1,\ldots,t_l)$ and $\neg R_k(t_1,\ldots,t_l)$ at the same time, then such conjunct is removed from the sentence as false in every ${\rm L}$\nobreakdash-structure.
Thus, we obtain an equivalent sentence~$\neg \varphi_5$ in PDF, in which each conjunct contains conditions for satisfaction or dissatisfaction of all possible sets of its variables to all necessary predicates~${R_k \in {\rm L_\infty}}$.

\medskip

--- If all conjuncts of the current sentence are deleted after performing steps 3--5, then the original sentence~$\varphi$ is true for all ${\rm L}$\nobreakdash-structures and, therefore, belongs to $Th_{ \forall}(\textbf{\em K})$.
In this case, the algorithm outputs <<YES>> and finishes.

--- Otherwise, we get the sentence~$\neg \varphi_5$ in PDF, which equivalent to the sentence~$\neg \varphi$, and each conjunct of the sentence~$\neg \varphi_5$ defines the condition for the existence of some substructure of the language~${\rm L}$.
The algorithm moves to step~6.

\medskip

{\it Step~6.}
The algorithm looks sequentially all conjuncts of the sentence~$\neg \varphi_5$.
For the current conjunct~$\psi_r$ the algorithm constructs the ${\rm L}$\nobreakdash-structure, which is defined by its condition, then checks this ${\rm L}$\nobreakdash-structure for membership in the class~$\textbf{\em K}$.

\medskip

--- If the constructed ${\rm L}$\nobreakdash-structure belongs to the class~$\textbf{\em K}$, then the algorithm outputs <<NO>> and finishes.

--- If the constructed ${\rm L}$\nobreakdash-structure doesn't belong to the class~$\textbf{\em K}$, then the algorithm goes to the next conjunct.

--- If all conjuncts of the sentence~$\neg \varphi_5$ are examined and there are no models from the class~$\textbf{\em K}$ for all of them, then the algorithm outputs <<YES>> and finishes.

\medskip

For the current conjunct~$\psi_r$ of the sentence~$\neg \varphi_5$ the algorithm constructs 
\linebreak
a~$q$\nobreakdash-element ${\rm L}$\nobreakdash-structure ${A_r = \langle A, {\rm L} \rangle}$, the elements of which correspond one-to-one to the variables of the conjunct.

Since the set of minimal forbidden substructures of the class~$\textbf{\em K}$ consists of finite ${\rm L}$\nobreakdash-structures and is recursive, there is a~procedure that allows to find out whether an arbitrary finite ${\rm L}$\nobreakdash-structure belongs to this set.
Using this procedure, one can determine a~set of all axioms $\{\theta\}$ whose negations correspond one-to-one to minimal forbidden substructures of the class~$\textbf{\em K}$ having at most $q$~elements.

To make sure that the ${\rm L}$\nobreakdash-structure~$A_r$ belongs to the class~$\textbf{\em K}$, one need to check whether it doesn't have forbidden substructures corresponding to the set of sentences~$\{\theta\}$, i.\,e. whether there is any existential sentence~$\neg \theta$ that is true in the ${\rm L}$\nobreakdash-structure~$A_r$.

To do this, for each axiom~$\theta$ containing $n$~variables $(n \leqslant q)$ all possible correspondences between the variables ${\{x_1, x_2, ..., x_n\}}$ of the axiom~$\theta$ and the elements ${\{1,2,...,q\}}$ of the ${\rm L}$\nobreakdash-structure~$A_r$ are considered (see table~\ref{tabular:tabl1} for ${n=3}$, ${q=4}$).
For each such correspondence, the truth of~$\neg \theta$ is checked in the ${\rm L}$\nobreakdash-structure~$A_r$.

\begin{table}[h]
\caption{Correspondences between the variables of the axiom~$\theta$ and the elements of the ${\rm L}$\nobreakdash-structure.}
\label{tabular:tabl1}
\begin{center}
\begin{tabular}{|c|c|c|c|}
\hline
\ \ \ №\ \ \  & $\ \ \ \ \ x_1\ \ \ \ \ $ & $\ \ \ \ \ x_2\ \ \ \ \ $ & $\ \ \ \ \ x_3\ \ \ \ \ $ \\
\hline
\textbf{1} & 1 & 2 & 3 \\
\hline
\textbf{2} & 1 & 2 & 4 \\
\hline
\textbf{3} & 1 & 3 & 2 \\
\hline
\textbf{4} & 1 & 3 & 4 \\
\hline
\textbf{5} & 1 & 4 & 2 \\
\hline
\textbf{6} & 1 & 4 & 3 \\
\hline
... & ... & ... & ... \\
\vspace{-4mm} & \  & \  & \  \\
\hline
\textbf{23} & 4 & 3 & 1 \\
\hline
\textbf{24} & 4 & 3 & 2 \\
\hline

\end{tabular}
\end{center}
\end{table}

\medskip

--- If at least one of the sentences~$\{\neg \theta\}$ is true in the ${\rm L}$\nobreakdash-structure~$A_r$, then it doesn't belong to the class~$\textbf{\em K}$.
In this case, the algorithm goes to the next conjunct of the sentence~$\neg \varphi_5$.

--- If all sentences~$\{\neg \theta\}$ are false on the ${\rm L}$\nobreakdash-structure~$A_r$, then it doesn't contain forbidden substructures for the class~$\textbf{\em K}$ and, therefore, belongs to the class~$\textbf{\em K}$.
Thus, for the sentence~$\neg \varphi$ a~model from the class~$\textbf{\em K}$ is constructed.

\medskip

As a~result, the algorithm answers the question whether the universal sentence~$\varphi$ belongs to the theory~$Th_{\forall}(\textbf{\em K})$.
\end{proof}

\end{document}